%% file: main.tex
\documentclass[a4paper,onecolumn,superscriptaddress,10pt]{compositionalityarticle}
\pdfoutput=1
\usepackage[margin=3cm]{geometry}
\input{core}
\bibliography{references}

\newcommand{\A}{{\cl A}}
\newcommand{\C}{{\cl C}}
\newcommand{\E}{{\cl E}}
\newcommand{\J}{{\cl J}}
\newcommand{\M}{{\cl M}}
\newcommand{\V}{{\cl V}}
\newcommand{\nat}[1]{\text{$#1$ nat.}}
\newcommand{\unnat}{_{\tx{unnat}}}

\newcommand{\authorsforheader}{Arkor}
\newcommand{\compositionalityIssue}{7}
\newcommand{\compositionalityVolume}{8}
\newcommand{\paperdoinumber}{10.46298/compositionality-\compositionalityVolume-\compositionalityIssue}
\newcommand{\paperdoi}{https://doi.org/\paperdoinumber}
\newcommand{\receiveddate}{2026-01-28}
\newcommand{\accepteddate}{2026-11-04}

\title{Magmal characterisations of cocartesian categories}
\author{Nathanael Arkor}
\orcid{0000-0002-4092-7930}
\address{Department of Software Science, Tallinn University of Technology, Estonia}
\thanks{The author was supported by a departmental postdoctoral grant from the Department of Software Science at Tallinn University of Technology.}
\date{21 November 2026}

\begin{document}

\begin{abstract}
  We present a survey of characterisations of cocartesian categories in terms of monoidal categories -- and, more generally, magmal categories -- satisfying additional properties. In particular, we show that the following are equivalent for a unital magmal category $(\M, \otimes)$, sharpening several classical characterisations.
  \begin{enumerate}[itemsep=0mm]
    \item $(\M, \otimes)$ is cocartesian monoidal.
    \item Every object of $\M$ admits the structure of a unital magma with respect to $\otimes$, such that every morphism is a homomorphism, and a single compatibility condition holds between the magma structures and $\otimes$.
    \item The tensor product functor ${\otimes} \colon \M \times \M \to \M$ admits a right adjoint.
  \end{enumerate}
\end{abstract}

\maketitle

\section{Introduction}

Every cocartesian category -- that is, every category with finite coproducts -- is a monoidal category, the tensor product $A \otimes B$ being given by the binary coproduct $A + B$, and the unit $I$ being given by the initial object $0$. Conversely, it is natural to ask how we might characterise the cocartesian monoidal categories amongst the monoidal categories, in a manner amenable to verification in particular examples. The purpose of this paper is to provide a convenient reference for several characterisation results for cocartesian monoidal categories, as appearing in \cref{characterisations}; as well as to show that the sharpest of the characterisations, \eqref{sharpest-characterisation}, is essentially optimal, in the sense that each of its assumptions really is necessary\footnotemark{}. In fact, for the sake of generality, we work primarily not with monoidal categories, but \emph{colax unital magmal} categories, which differ from monoidal categories in not imposing an associativity constraint and in not requiring invertibility of the unitors $\lambda$ and $\rho$.
\footnotetext{In fact, there is one further mild generalisation: naturality of either one (but not both) of the unitors $\lambda$ or $\rho$ may be dropped (see \cref{further-generalisation}).}%

\begin{theorem}
    \label{characterisations}
    For a colax unital magmal category $(\M, \otimes, I, \lambda, \rho)$, the following are equivalent.
    \begin{enumerate}[label=\alph*.]
        \item \label{cocartesian} $(\M, \otimes, I, \lambda, \rho)$ is cocartesian.
        \item \label{uniqueness-characterisation} $\M$ is symmetric monoidal, every object of $\M$ admits a unique commutative monoid structure, and every morphism in $\M$ is a homomorphism with respect to the unique monoid structures.
        \item \label{symmetry-characterisation} $(\M, \otimes)$ is symmetric monoidal, every object $A \in \M$ admits a unital magma structure $(\mu_A, \eta_A)$ for which every morphism is a homomorphism, and the diagram \eqref{symmetry-equation} commutes for each $A, B \in \M$ (where $A \otimes \sigma_{A, B} \otimes B$ denotes the canonical morphism induced by the braiding $\sigma_{A, B}$ and by associativity).\footnote{Commutativity of the diagram \eqref{symmetry-equation} is precisely the condition that $\mu$ is a magmal natural transformation; in particular, this holds if (and only if) $\mu$ is moreover a monoidal natural transformation: see \cref{magmality-of-mu}.}
        \begin{equation}
        \label{symmetry-equation}
		\begin{tikzcd}
			{(A \otimes A) \otimes (B \otimes B)} && {(A \otimes B) \otimes (A \otimes B)} \\
			& {A \otimes B}
			\arrow["{A \otimes \sigma_{A, B} \otimes B}", from=1-1, to=1-3]
			\arrow["{\mu_A \otimes \mu_B}"', from=1-1, to=2-2]
			\arrow["{\mu_{A \otimes B}}", from=1-3, to=2-2]
		\end{tikzcd}
        \end{equation}
        \item \label{sharpest-characterisation} Every object $A \in \M$ admits a unital magma structure $(\mu_A, \eta_A)$ for which every morphism is a homomorphism, and the diagram \eqref{quasi-symmetry-equation} commutes for each $A, B \in \M$.
        \begin{equation}
        \label{quasi-symmetry-equation}
		\begin{tikzcd}[column sep=large]
			{(A \otimes I) \otimes (I \otimes B)} && {(A \otimes B) \otimes (A \otimes B)} \\
			{A \otimes B} && {A \otimes B}
			\arrow["{(A \otimes \eta_B) \otimes (\eta_A \otimes B)}", from=1-1, to=1-3]
			\arrow["{\mu_{A \otimes B}}", from=1-3, to=2-3]
			\arrow["{\rho_A \otimes \lambda_B}", from=2-1, to=1-1]
			\arrow[equals, from=2-1, to=2-3]
		\end{tikzcd}
        \end{equation}
        \item \label{adjoint-characterisation} ${\otimes} \colon \M \times \M \to \M$ admits a right adjoint, and $\lambda$ and $\rho$ are invertible.
    \end{enumerate}
\end{theorem}

\begin{remark}
	\label{reformulations}
    There are several equivalent ways to express the existence of structure on every object for which every morphism is a homomorphism. For instance, the following are equivalent for a symmetric monoidal category $\M$. (There are two sets of equivalent statements, given by either including in every statement the words inside square brackets, or excluding them.)
    \begin{itemize}
		\item Every object in $\M$ admits a [unique] commutative monoid structure, for which every morphism in $\M$ is furthermore a homomorphism.
		\item The forgetful functor from the category of commutative monoids in $\M$ admits an [invertible] section.
    \item The identity functor on $\M$ admits a commutative monoid structure in $[\M, \M]$ (with respect to the pointwise symmetric monoidal structure on $[\M, \M]$) [and this is furthermore the unique commutative monoid structure on the identity functor in the category $[\M, \M]\unnat$ of endofunctors and \emph{unnatural} transformations].
    \end{itemize}
    While the first perspective is perhaps the most elementary, in the remainder of this paper we shall primarily emphasise the third perspective, leaving the straightforward reformulations to the reader.
\end{remark}

\begin{remark}
	Assuming that the unitors $\lambda$ and $\rho$ in \cref{characterisations} are invertible (which will typically be the case in practice), the commutativity condition \eqref{quasi-symmetry-equation} may be replaced by the following, more familiar, triangle.
	\[\begin{tikzcd}[column sep=large]
		{(A \otimes I) \otimes (I \otimes B)} && {(A \otimes B) \otimes (A \otimes B)} \\
		& {A \otimes B}
		\arrow["{(A \otimes \eta_B) \otimes (\eta_A \otimes B)}", from=1-1, to=1-3]
		\arrow["{\rho\inv_A \otimes \lambda\inv_B}"', from=1-1, to=2-2]
		\arrow["{\mu_{A \otimes B}}", from=1-3, to=2-2]
	\end{tikzcd}\qedshift\]
\end{remark}

\begin{remark}
	Naturally, every characterisation of \emph{cocartesian} monoidal categories is equivalently a characterisation of \emph{cartesian} monoidal categories, by considering the opposite of the category in question. Such characterisations respectively involve \emph{co}unital \emph{co}magmas, and a \emph{left} adjoint to ${\otimes} \colon \M \times \M \to \M$.
\end{remark}

\subsection{Prior work}

The characterisations of \cref{characterisations} might be viewed as categorical folklore, similar characterisations having been frequently rediscovered in the literature. However, the present paper was motivated by the discovery that locating references for these characterisations in complete generality is surprisingly difficult: to our knowledge, conditions \mbox{(\ref{sharpest-characterisation} \& \ref{adjoint-characterisation})} do not appear in the literature without also being accompanied by further redundant assumptions. Furthermore, more concerningly, some references impose \emph{fewer} assumptions than are necessary to obtain a characterisation. It thus seemed worthwhile to set the record straight.

To provide some historical context, we shall explain, to the extent we have been able to ascertain, the precedent for the two flavours of characterisations presented in \cref{characterisations}.

\subsubsection{Monoidal characterisations of cocartesian categories}

The earliest characterisation of cocartesian monoidal categories in terms of the existence of monoid structure on each object is usually credited to a \citeyear{fox1976coalgebras} paper of \textcite{fox1976coalgebras}, the term \emph{Fox's theorem} subsequently often being associated to various characterisations of this form. \citeauthor{fox1976coalgebras} proved that the construction of the category of commutative monoids in a symmetric monoidal category provides a right adjoint to the inclusion of the category of cocartesian monoidal categories into the category of symmetric monoidal categories (and strict\footnotemark{} symmetric monoidal functors).
\footnotetext{The result holds also for other varieties of symmetric monoidal functor (\cf~\cite[Corollary~3.19]{pisani2014sequential}).}%
\begin{equation}
\begin{tikzcd}[column sep=large]
	{\b{CocartCat}_s} & {\b{SymMonCat}_s}
	\arrow[""{name=0, anchor=center, inner sep=0}, shift left=2, hook, from=1-1, to=1-2]
	\arrow[""{name=1, anchor=center, inner sep=0}, "{\b{CMon}}", shift left=2, from=1-2, to=1-1]
	\arrow["\dashv"{anchor=center, rotate=-90}, draw=none, from=0, to=1]
\end{tikzcd}
\end{equation}
It therefore follows that a symmetric monoidal category is cocartesian if and only if every object admits a unique commutative monoid structure for which every morphism is a homomorphism: this is (\ref{cocartesian} $\iff$ \ref{uniqueness-characterisation}) in \cref{characterisations}.

That the assumption of uniqueness of the monoid structure on each object may be replaced by suitable compatibility conditions (\eg condition \eqref{symmetry-equation} or \eqref{quasi-symmetry-equation}), and that it suffices to consider unital magmas rather than commutative monoids, appears to have been frequently rediscovered (\cf{}~\cites[\S1]{corradini1999algebraic}[\S6.4]{mellies2009categorical}). That symmetry of the monoidal category is not necessary has been less frequently observed but appears, for instance, in \cite[Proposition~16]{mellies2009categorical}.

The most general prior characterisation was given by Jeff Egger in an email to the categories mailing list\footnote{This was brought to our attention by a remark in a preprint of \textcite{pisani2014sequential} during the preparation of this paper.} on the 15\xth{} of March 2013, in which Egger observed that the commutativity of the diagram \eqref{quasi-symmetry-equation} is the only necessary condition on the unital magma structures to guarantee uniqueness. Our characterisation is thus a mild generalisation of Egger's, demonstrating that neither associativity of the monoidal structure in question, nor invertibility of the unitors, is necessary.

On the other hand, in \cite[Remark~1.3.(iii)]{carboni1987cartesian} and elsewhere, it has been claimed that a symmetric monoidal category is cocartesian if and only if every object admits a commutative monoid structure for which every morphism is a homomorphism (in other words, that imposing uniqueness of the monoid structures, or equivalently the commutativity of the diagrams \eqref{symmetry-equation} or \eqref{quasi-symmetry-equation}, is unnecessary). We shall show in \cref{necessity} that this is not true.

Applications of monoidal characterisations of (co)cartesian categories are widespread in the literature: for instance, since they replace universal properties with purely equational properties, they may be invoked to present graphical calculi for categories with (co/bi)products~\cite[\S6]{selinger2010survey}; are a common tool in the categorical semantics of linear logic~\cite{mellies2009categorical}; and motivate important variations of monoidal categories, such as Markov categories in categorical probability theory~\cite{fritz2020synthetic}. The minimal characterisations in \cref{characterisations} therefore promise to be genuinely useful in reducing the number of verifications necessary to establish (co)cartesianness in practice.

\subsubsection{Adjoint characterisations of cocartesian categories}

By definition, a cocartesian monoidal category is a monoidal category $(\M, \otimes, I)$ for which the tensor product functor ${\otimes} \colon \M \times \M \to \M$ is left-adjoint to the diagonal functor $\Delta \colon \M \to \M \times \M$, and for which the unit object $I \in \M$ is initial (equivalently, the functor $I \colon 1 \to \M$ picking out the unit is left-adjoint to the unique functor $\unit \colon \M \to 1$). In \cite[\S3]{street2014kan}, \citeauthor{street2014kan} observed that it suffices for ${\otimes} \colon \M \times \M \to \M$ to be left-adjoint to \emph{some} functor (not necessarily the diagonal), in addition to $I$ being initial. This observation is useful for enriched category theory, as it provides a natural generalisation of cocartesian monoidal categories to enriched categories, when a diagonal functor may not exist (since the base of enrichment need not admit products).

On the 3\textsuperscript{rd} of January 2021, unaware of \citeauthor{street2014kan}'s characterisation, the author asked on the question-and-answer site MathOverflow for a reference for monoidal categories in which the tensor product functor admits a left adjoint.\footnotemark{} In response, Qiaochu Yuan and Tim Campion together proved that these are precisely the cartesian monoidal categories, thus improving upon \citeauthor{street2014kan}'s characterisation by dropping the assumption on the unit. As far as we are aware, this observation has not appeared previously in the literature. We give a new proof of this fact in \cref{an-adjoint-characterisation} by combining the observations of Street, Yuan, and Campion.
\footnotetext{\url{https://mathoverflow.net/questions/380302/monoidal-categories-whose-tensor-has-a-left-adjoint}}%

\begin{remark}
	Monoidal categories may be viewed as certain kinds of multicategories, namely the \emph{representable multicategories}~\cite{hermida2000representable}. Under this relationship, cocartesian monoidal categories correspond precisely to the representable \emph{cocartesian multicategories} (called \emph{sequential} in \cite{pisani2014sequential}). In \cite[Proposition~3.6]{pisani2014sequential}, \citeauthor{pisani2014sequential} gave a characterisation of the cocartesian multicategories in terms of symmetric multicategories admitting a \emph{central unital magma}. When specialised to representable multicategories, this recovers the characterisation in \eqref{symmetry-characterisation}~\cite[Proposition~3.21]{pisani2014sequential}. It would be interesting to see whether this multicategorical characterisation can be generalised to the same extent as \cref{characterisations} (for instance, by dropping the assumption of symmetry on the multicategory).
\end{remark}

\subsection{Acknowledgements}

We thank Qiaochu Yuan and Tim Campion for their insightful answers to the author's questions, and for providing the original proof of the characterisation of monoidal categories with adjoint tensor product. We also thank Tim Campion, Vikraman Choudhury, Rose Kudzman-Blais, and the anonymous reviewer for helpful comments on the paper.
We thank Peter Selinger for pointing out that a unit compatibility condition is derivable (\cref{eta-I-is-1}), which prompted us to search for the minimal conditions; and Robin Houston for providing several references.

\section{Naturally weak adjoints and naturally weak colimits}

Our proof strategy for establishing the characterisation of cocartesian monoidal categories in terms of unital magmas is based on observing the existence of certain canonical idempotents, whose splittings exhibit the promised cocartesian structure. For this, it will be helpful to introduce the preliminary concepts of \emph{naturally weak relative adjoint} and \emph{naturally weak colimit}. However, we note that these concepts are used only as a technical tool in certain proofs, and are not necessary to follow the remainder of the paper. The reader should therefore feel welcome to skip to \cref{magmal-categories-and-magmas} and refer back to this section only as necessary.

A functor $L \colon \A \to \C$ is a \emph{weak left adjoint} of a functor $R \colon \C \to \A$ when there is a family of surjections $\C(LX, Y) \epito \A(X, RY)$, natural in $X \in \A$ and $Y \in \C$~\cite[\S1]{kainen1971weak}. Assuming the axiom of choice, every surjection admits a section; in the absence of choice, it is natural to instead ask that a weak left adjoint comes equipped with a choice of section for each $X \in \A$ and $Y \in \C$. Note that, even then, there is no guarantee that the choices of sections assemble into a natural transformation. We introduce the following definition to capture a notion of weak adjoint satisfying this stronger property. In fact, for what follows, it will be useful to work with a weak analogue of the more general concept of \emph{relative} adjoint functors~\cite{ulmer1968properties}.

\begin{definition}
	\label{naturally-weak-relative-adjunction}
	A \emph{naturally weak relative adjunction} comprises a diagram of functors and a natural transformation as follows,
	\[\begin{tikzcd}
		& \C \\
		\A && \E
		\arrow["R", from=1-2, to=2-3]
		\arrow["L", from=2-1, to=1-2]
		\arrow[""{name=0, anchor=center, inner sep=0}, "J"', from=2-1, to=2-3]
		\arrow["\eta", between={0.3}{0.7}, Rightarrow, from=0, to=1-2]
	\end{tikzcd}\]
	such that the induced natural transformation
	\[\C(L{-}, {-}) \xtto{R_{L{-}, {-}}} \E(RL{-}, R{-}) \xtto{\E(\eta, R{-})} \E(J{-}, R{-})\]
	admits a chosen section $\flat \colon \E(J{-}, R{-}) \tto \C(L{-}, {-})$. A \emph{relative adjunction} is a naturally weak relative adjunction for which $\flat$ is not just a section, but furthermore an inverse.
\end{definition}

Note that, in contrast to adjoints, naturally weak adjoints are not uniquely determined: a given functor may have many non-isomorphic naturally weak adjoints.

If we are given a naturally weak relative adjoint, we should like to know when we have an actual relative adjoint. The following lemma shows that it suffices for a certain idempotent to be split (\cf~\cites[Exercise~IV.1.4]{maclane1998categories}{street1996consequences}).

\begin{lemma}
	\label{relative-adjoint-via-splitting}
	Suppose we are given a naturally weak relative adjunction as in \cref{naturally-weak-relative-adjunction}. Then, for each object $X \in \A$, the object $LX \in \C$ is equipped with an idempotent $\flat_{X, LX}(\eta_X)$. $R$ admits a $J$-relative left adjoint if and only if each such idempotent splits, in which case the splittings exhibit the left adjoint. In particular, $\eta$ exhibits $L$ as left adjoint to $R$ relative to $J$ if and only if $\flat_{X, LX}(\eta_X)$ is the identity for each $X \in \A$.
\end{lemma}

\begin{proof}
	The natural section--retraction pairs defined by a naturally weak relative adjunction induce an idempotent on each $\C(LX, {-})$, which in turn induces an idempotent $\flat_{X, LX}(\eta_X)$ on $LX$ by the Yoneda lemma. Since $\E(JX, R{-})$ splits the idempotent $\C(\flat_{X, LX}(\eta_X), {-})$, and splittings of idempotents are essentially unique, if $\flat_{X, LX}(\eta_X)$ splits as $LX \xto r SX \xto s LX$, then $\C(SX, {-}) \iso \E(JX, R{-})$. Conversely, if $R$ admits a $J$-relative left adjoint $S$, each $SX$ splits $\flat_{X, LX}(\eta_X)$, again by the Yoneda lemma.
\end{proof}

Just as we may define colimits in terms of (relative) adjoints, so too may we define the notion of \emph{naturally weak colimit} in terms of naturally weak (relative) adjoints, providing an analogous strengthening of the notion of \emph{weak colimit}~\cite[\S3]{kainen1971weak}. A weak colimit satisfies the \emph{existence property} of mediating morphisms from a colimit, without necessarily satisfying the \emph{uniqueness property}. A naturally weak colimit furthermore satisfies the property that the choices of mediating morphisms are natural in their codomain.

\begin{definition}
	Let $D \colon \J \to \M$ be a functor. A \emph{naturally weak colimit} for $D$ is a naturally weak left adjoint to the diagonal functor $\Delta \colon \M \to \M^\J$, relative to the functor $D \colon 1 \to \cl M^\J$.
	\[\begin{tikzcd}
		& \M \\
		1 && {\M^\J}
		\arrow["\Delta", from=1-2, to=2-3]
		\arrow[dashed, from=2-1, to=1-2]
		\arrow[""{name=0, anchor=center, inner sep=0}, "D"', from=2-1, to=2-3]
		\arrow["\copi", between={0.3}{0.7}, Rightarrow, from=0, to=1-2]
	\end{tikzcd}\qedshift\]
\end{definition}

Specialising \cref{relative-adjoint-via-splitting} to naturally weak colimits gives a sufficient condition for the existence of colimits. We note that this proof strategy is not original: it appears, for instance, in the construction of flexible limits in 2-categories~\cite[Proposition~4.8]{bird1989flexible} (albeit without explicitly identifying the notion of naturally weak colimit).

\begin{corollary}
	\label{colimit-from-splitting-idempotent}
	Let $D \colon \J \to \M$ be a functor admitting a naturally weak colimit $(C, \copi)$. Then $C$ is equipped with an idempotent $e$, which splits if and only if $D$ admits a colimit, in which case the splitting exhibits the colimit. In particular, $(C, \copi)$ exhibits the colimit of $D$ if and only if $e$ is the identity.
\end{corollary}

\begin{proof}
	Immediate from \cref{relative-adjoint-via-splitting}.
\end{proof}

To conclude the section, we recall an elementary property of idempotents that will be useful in what follows.

\begin{lemma}
    \label{idempotent-retraction}
    An idempotent $e$ admits a retraction if and only if $e = 1$.
\end{lemma}

\begin{proof}
    One direction is trivial. For the other, assume the existence of a retraction $r$. Then we have:
    \[e = (re)e = r(ee) = re = 1 \qedhere\]
\end{proof}

\section{Magmal categories and magmas}
\label{magmal-categories-and-magmas}

We may now proceed with presenting our promised characterisation of cocartesian monoidal categories. We begin by introducing the notion of \emph{colax unital magmal category} (mildly generalising the unital magmal categories of \cite[\S2.1 \& \S2.5]{davydov2007nuclei}).

\begin{definition}
    A \emph{magmal category} comprises a category $\M$ equipped with a functor ${{\otimes} \colon \M \times \M \to \M}$, the \emph{tensor product}. A magmal category $(\M, \otimes)$ is \emph{colax unital} when equipped with an object $I \in \M$, the \emph{unit}, and natural transformations $\lambda \colon \ph \tto I \otimes \ph$ and $\rho \colon \ph \tto \ph \otimes I$, the \emph{left-} and \emph{right-unitor} respectively, satisfying $\lambda_I = \rho_I \colon I \to I \otimes I$. A colax unital magmal category is \emph{unital} if the unitors $\lambda$ and $\rho$ are invertible.
\end{definition}

If a magmal category admits two unital structures $(I, \lambda, \rho)$ and $(I', \lambda', \rho')$, their unit objects are canonically isomorphic:
\begin{equation}
	I \xto{\rho'_I} I \otimes I' \xto{\lambda\inv_{I'}} I' \hspace{8em}
    I \xto{\lambda'_I} I' \otimes I \xto{\rho\inv_{I'}} I'
\end{equation}

Just as we may consider monoids in a monoidal category, we may consider [unital] magmas in [colax unital] magmal categories.

\begin{definition}
    Let $(\M, \otimes, I)$ be a colax unital magmal category. A \emph{unital magma} therein comprises
    \begin{enumerate}
        \item an object $A \in \M$;
        \item a morphism $\eta \colon I \to A$;
        \item a morphism $\mu \colon A \otimes A \to A$,
    \end{enumerate}
    rendering the following diagram commutative.
	\[\begin{tikzcd}
		{I \otimes A} & {A \otimes A} & {A \otimes I} \\
		A & A & A
		\arrow["{\eta \otimes A}", from=1-1, to=1-2]
		\arrow["\mu"{description}, from=1-2, to=2-2]
		\arrow["{A \otimes \eta}"', from=1-3, to=1-2]
		\arrow["{\lambda_A}", from=2-1, to=1-1]
		\arrow[equals, from=2-1, to=2-2]
		\arrow[equals, from=2-2, to=2-3]
		\arrow["{\rho_A}"', from=2-3, to=1-3]
	\end{tikzcd}\]
    A \emph{homomorphism} from $(A, \mu_A, \eta_A)$ to $(B, \mu_B, \eta_B)$ is a morphism $f \colon A \to B$ rendering the following diagram commutative.
    \[\begin{tikzcd}
    	{A \otimes A} && {B \otimes B} \\
    	A && B \\
    	& I
    	\arrow["{f \otimes f}", from=1-1, to=1-3]
    	\arrow["{\mu_A}"', from=1-1, to=2-1]
    	\arrow["{\mu_B}", from=1-3, to=2-3]
    	\arrow["f"{description}, from=2-1, to=2-3]
    	\arrow["{\eta_A}", from=3-2, to=2-1]
    	\arrow["{\eta_B}"', from=3-2, to=2-3]
    \end{tikzcd}\]
    Unital magmas in $(\M, \otimes, I)$ and their homomorphisms form a category equipped with a forgetful functor to $\M$.
\end{definition}

Observe that a [colax unital] magmal structure on $\M$ equips the endofunctor category $[\M, \M]$ with [colax unital] magmal structure pointwise: explicitly, given endofunctors $F, G \colon \M \to \M$, the tensor product $(F \otimes G)\ph \defeq F\ph \otimes G\ph$; and the unit is the constant functor $\u I \colon \M \to \M$ on $I \in \M$.
The relevance of unital magmas to cocartesian monoidal categories is illustrated by the following observation.

\begin{lemma}[\cite{fox1976coalgebras}]
    \label{cartesian-implies-unique-commutative-comonoid}
    If $\M$ is a cocartesian monoidal category, then every object of $\M$ admits a unique unital magma structure for which every morphism is a homomorphism. Furthermore, this structure is associative and commutative.
\end{lemma}

\begin{proof}
    For each object $A \in \M$, there is a unique morphism $[]_A \colon 0 \to A$ by initiality. Furthermore, the universal property of $A + A$ uniquely determines a multiplication.
	\[\begin{tikzcd}[row sep=large]
		{0 + A} & {A + A} & {A + 0} \\
		& A
		\arrow["{[]_A + A}", from=1-1, to=1-2]
		\arrow["{[[]_A, A]}"', from=1-1, to=2-2]
		\arrow["{[A, A]}"{description}, dashed, from=1-2, to=2-2]
		\arrow["{A + []_A}"', from=1-3, to=1-2]
		\arrow["{[A, []_A]}", from=1-3, to=2-2]
	\end{tikzcd}\]
		It follows directly from the universal property of the coproducts that every morphism is a homomorphism and that the magma structures are associative and commutative.
\end{proof}

The purpose of the next few sections is to show that \cref{cartesian-implies-unique-commutative-comonoid} admits several converses, even after dropping the assumption of uniqueness of the unital magma structures.

\section{Semicocartesian categories}

Our first step is to identify conditions ensuring that the unit of a magmal category is initial, a property for which it is useful to have a name.

\begin{definition}
	A colax unital magmal category $(\M, \otimes, I, \lambda, \rho)$ is \emph{semicocartesian} if $I$ is initial.
\end{definition}

We recall a simple observation regarding the existence of initial objects in a category.

\begin{lemma}[{\cite[Proposition~4.2]{marmolejo1995ultraproducts}}]
    \label{initiality-via-splitting}
    Let $\M$ be a category equipped with an object $I$ and a natural transformation ${\eta \colon \u I \tto 1_\M}$. Then $\eta_I$ is idempotent, and splits if and only if $\M$ admits an initial object, in which case the splitting is initial. In particular, $I$ is initial if and only if $\eta_I = 1_I$.
\end{lemma}

\begin{proof}
    The natural transformation $\eta$ exhibits a naturally weak initial object, since we trivially have a section--retraction pair $\unit \colon \M(I, {-}) \leftrightarrows 1 \cocolon \eta$, from which the result follows by \cref{colimit-from-splitting-idempotent}.
\end{proof}

The following lemma connects this observation to the existence of unital magma structure; in particular, the assumption holds if the identity functor on $\M$ admits a unital magma structure in~$[\M, \M]$.

\begin{lemma}
    \label{eta-I-is-1}
    Let $(M, \otimes, I, \lambda, \rho)$ be a colax unital magmal category equipped with a natural transformation $\eta \colon \u I \tto 1_\M$. Suppose that $I$ admits a unital magma structure $(\mu_I, \eta_I)$. Then $I$ is initial.
\end{lemma}

\begin{proof}
    The following diagram commutes using the right unit law for $I$, and implies that $\eta_I$ admits a retraction $\mu_I \lambda_I$. Hence, by \cref{initiality-via-splitting,idempotent-retraction}, $I$ is initial.
	\[\begin{tikzcd}[column sep=large]
		I && I \\
		I & {I \otimes I} & {I \otimes I} \\
		& I
		\arrow["{\eta_I}", from=1-1, to=1-3]
		\arrow[equals, from=1-1, to=2-1]
		\arrow[""{name=0, anchor=center, inner sep=0}, "{\lambda_I}"{description}, from=1-1, to=2-2]
		\arrow[""{name=1, anchor=center, inner sep=0}, "{\lambda_I}", from=1-3, to=2-3]
		\arrow["{\rho_I}"{description}, from=2-1, to=2-2]
		\arrow["{I \otimes \eta_I}"{description}, from=2-2, to=2-3]
		\arrow["{\mu_I}", from=2-3, to=3-2]
		\arrow[equals, from=3-2, to=2-1]
		\arrow["{\text{$\lambda$\ nat.}}"{description}, draw=none, from=1, to=0]
	\end{tikzcd}\qedshift\]
\end{proof}

For the subsequent section, it will be useful to have the following mild generalisation of \cite[
Theorem~3.5]{gerhold2022categorial} from monoidal categories to unital magmal categories, which establishes that semicocartesian magmal categories admit canonical choices of coprojections.

\begin{proposition}
    \label{semicocartesian-characterisations}
    For a unital magmal category $(\M, \otimes, I, \lambda, \rho)$, there is a bijection between the following.
    \begin{enumerate}[label=\alph*.]
        \item Natural transformations $[] \colon \u I \tto 1_\M$ exhibiting $I$ as initial.
        \item Natural coprojections $\copi_1^{A, B} \colon A \to A \otimes B$ and $\copi_2^{A, B} \colon B \to A \otimes B$ such that $\copi_2^{I, A} \colon A \to I \otimes A$ is $\lambda_A$ and $\copi_1^{A, I} \colon A \to A \otimes I$ is $\rho_A$.
    \end{enumerate}
	If $\M$ is merely colax unital, then (a) $\implies$ (b).
\end{proposition}

\begin{proof}
    (a) $\implies$ (b). We define coprojections as follows, which trivially satisfy the given conditions since $[]_I = 1_I$.
    \begin{align*}
        A \xto{\rho_A} A \otimes I \xto{A \otimes []_B} A \otimes B \\
        B \xto{\lambda_B} I \otimes B \xto{[]_A \otimes B} A \otimes B
    \end{align*}

    (b) $\implies$ (a). We define a natural transformation
    \[I \xto{\copi_1^{I, A}} I \otimes A \xto{\lambda\inv_A} A\]
    which is the identity when $A = I$ by assumption, using that $\lambda_I = \rho_I$. Thus by \cref{initiality-via-splitting}, $I$ is initial.

	That (a $\tto$ b $\tto$ a) is the identity is trivial. That (b $\tto$ a $\tto$ b) is the identity follows from naturality of the coprojections.
	\[
	\begin{tikzcd}
		A & {A \otimes B} \\
		A & {A \otimes I}
		\arrow["{\copi_1^{A, B}}", from=1-1, to=1-2]
		\arrow[equals, from=2-1, to=1-1]
		\arrow["{\copi_1^{A, I}}"', from=2-1, to=2-2]
		\arrow["{A \otimes []_B}"', from=2-2, to=1-2]
	\end{tikzcd}
	\hspace{4em}
	\begin{tikzcd}
		{A \otimes B} & B \\
		{I \otimes B} & B
		\arrow["{\copi_2^{A, B}}"', from=1-2, to=1-1]
		\arrow["{[]_A \otimes B}", from=2-1, to=1-1]
		\arrow[equals, from=2-2, to=1-2]
		\arrow["{\copi_2^{I, B}}", from=2-2, to=2-1]
	\end{tikzcd}
	\qedshift
	\]
\end{proof}

\section{Cocartesian categories}

Having established sufficient conditions for initiality of the unit, our next step is to identify conditions ensuring that the tensor product of a magmal category is actually a coproduct.

\begin{definition}
	\label{coprojections}
    A colax unital magmal category $(\M, \otimes, I, \lambda, \rho)$ is \emph{cocartesian} if it is semicocartesian (\ie{} $I$ is initial) and, for each $A, B \in \M$, the cospan of \cref{semicocartesian-characterisations}
    \[A \xto{\rho_A} A \otimes 0 \xto{A \otimes []_B} A \otimes B \xfrom{[]_A \otimes B} 0 \otimes B \xfrom{\lambda_B} B\]
    exhibits a binary coproduct.
\end{definition}

In this case, by the following lemma, it suffices to drop the prefix \emph{colax unital}.

\begin{lemma}
	If a colax unital magmal category is cocartesian, then it is necessarily unital (\ie{} the unitors are invertible).
\end{lemma}

\begin{proof}
	The following diagram commutes by the existence property of the mediating morphism.
	\[\begin{tikzcd}
		A & {0 \otimes A} \\
		& A
		\arrow["{\lambda_A}", from=1-1, to=1-2]
		\arrow[equals, from=1-1, to=2-2]
		\arrow["{\big[[]_A, 1_A\big]}", from=1-2, to=2-2]
	\end{tikzcd}\]
	The following diagram commutes by the uniqueness property of the mediating morphism, by precomposing by $[]_{0 \otimes A} \colon 0 \to 0 \otimes A$ and $\lambda_A \colon A \to 0 \otimes A$.
	\[\begin{tikzcd}
		{0 \otimes A} & A \\
		& {0 \otimes A}
		\arrow["{\big[[]_A, 1_A\big]}", from=1-1, to=1-2]
		\arrow[equals, from=1-1, to=2-2]
		\arrow["{\lambda_A}", from=1-2, to=2-2]
	\end{tikzcd}\]
	Invertibility of the right-unitor $\rho$ follows symmetrically.
\end{proof}

As discussed in the introduction, every cocartesian category is a monoidal category. It is relevant to observe that the associator in a cocartesian monoidal category is uniquely determined by the universal property of the coproducts: in other words, a cocartesian magmal category is monoidal in a unique way.

\begin{lemma}
    \label{unique-associator}
    For every cocartesian magmal category, there is a unique associator $\{ \alpha_{A, B, C} \colon (A \otimes B) \otimes C \to A \otimes (B \otimes C) \}_{A, B, C}$.
\end{lemma}

\begin{proof}
	The existence of an associator $\alpha$ is classical~\cite[\S IV.2]{eilenberg1966closed}, so it remains to establish uniqueness.

    First, observe that the following diagram commutes using initiality.
    \[\begin{tikzcd}
    	{(A \otimes B) \otimes 0} && {(A \otimes B) \otimes C} \\
    	{A \otimes B} & {A \otimes (B \otimes 0)} \\
    	{A \otimes 0} && {A \otimes (B \otimes C)}
    	\arrow["{(A \otimes B) \otimes []_C}", from=1-1, to=1-3]
    	\arrow["{\alpha_{A, B, 0}}"{description}, from=1-1, to=2-2]
    	\arrow["{\alpha_{A, B, C}}", from=1-3, to=3-3]
    	\arrow["{\rho_{A \otimes B}}", from=2-1, to=1-1]
    	\arrow["{A \otimes \rho_B}"{description}, from=2-1, to=2-2]
    	\arrow["{\nat\alpha}"{description}, draw=none, from=2-2, to=1-3]
    	\arrow["{A \otimes (B \otimes []_C)}"{description}, from=2-2, to=3-3]
    	\arrow["{A \otimes []_B}", from=3-1, to=2-1]
    	\arrow["{A \otimes []_{B \otimes C}}"', from=3-1, to=3-3]
    \end{tikzcd}\]
    Precomposing $\rho_A$, the top path is precisely the coprojection $A \to (A \otimes B) \otimes C$, while the bottom path is precisely the coprojection $A \to A \otimes (B \otimes C)$.

    Next, observe that the following diagram commutes.
    \[\begin{tikzcd}
      {(A \otimes B) \otimes 0} &&& {(A \otimes B) \otimes C} \\
      {A \otimes B} & {A \otimes (B \otimes 0)} && {A \otimes (B \otimes C)} \\
      {0 \otimes B} & {0 \otimes (B \otimes 0)} \\
      B & {B \otimes 0} & {B \otimes C} & {0 \otimes (B \otimes C)}
      \arrow["{(A \otimes B) \otimes []_C}", from=1-1, to=1-4]
      \arrow["{\alpha_{A, B, 0}}"{description}, from=1-1, to=2-2]
      \arrow["{\alpha_{A, B, C}}", from=1-4, to=2-4]
      \arrow["{\rho_{A \otimes B}}", from=2-1, to=1-1]
      \arrow["{A \otimes \rho_B}"{description}, from=2-1, to=2-2]
      \arrow["{\nat\alpha}"{description}, draw=none, from=2-2, to=1-4]
      \arrow["{A \otimes (B \otimes []_C)}"{description}, from=2-2, to=2-4]
      \arrow["{[]_A \otimes B}", from=3-1, to=2-1]
      \arrow["{0 \otimes \rho_B}"{description}, from=3-1, to=3-2]
      \arrow["{[]_A \otimes (B \otimes 0)}"{description}, from=3-2, to=2-2]
      \arrow["{[]_A \otimes (B \otimes []_C)}"{description}, from=3-2, to=2-4]
      \arrow["{\nat\lambda}"{description}, draw=none, from=3-2, to=4-3]
      \arrow["{0 \otimes (B \otimes []_C)}"{description}, from=3-2, to=4-4]
      \arrow["{\lambda_B}", from=4-1, to=3-1]
      \arrow["{\rho_B}"', from=4-1, to=4-2]
      \arrow["{\lambda_{B \otimes 0}}"{description}, from=4-2, to=3-2]
      \arrow["{B \otimes []_C}"', from=4-2, to=4-3]
      \arrow["{\lambda_{B \otimes C}}"', from=4-3, to=4-4]
      \arrow["{[]_A \otimes (B \otimes C)}"', from=4-4, to=2-4]
    \end{tikzcd}\]
    The top path is precisely the coprojection $B \to (A \otimes B) \otimes C$, while the bottom path is precisely the coprojection $B \to A \otimes (B \otimes C)$.

    Finally, observe that the following diagram commutes.
    \[\begin{tikzcd}
    	{0 \otimes C} &&& {(A \otimes B) \otimes C} \\
    	{B \otimes C} && {(0 \otimes B) \otimes C} \\
    	{0 \otimes (B \otimes C)} &&& {A \otimes (B \otimes C)}
    	\arrow["{[]_{A \otimes B} \otimes C}", from=1-1, to=1-4]
    	\arrow["{[]_B \otimes C}"', from=1-1, to=2-1]
    	\arrow["{[]_{0 \otimes B} \otimes C}"{description}, from=1-1, to=2-3]
    	\arrow["{\alpha_{A, B, C}}", from=1-4, to=3-4]
    	\arrow["{\lambda_B \otimes C}"{description}, from=2-1, to=2-3]
    	\arrow["{\lambda_{B \otimes C}}"', from=2-1, to=3-1]
    	\arrow["{([]_A \otimes B) \otimes C}"{description}, from=2-3, to=1-4]
    	\arrow["{\alpha_{0, B, C}}"{description}, from=2-3, to=3-1]
    	\arrow["{\nat\alpha}"{description}, draw=none, from=2-3, to=3-4]
    	\arrow["{[]_A \otimes (B \otimes C)}"', from=3-1, to=3-4]
    \end{tikzcd}\]
    Precomposing $\lambda_C$, the top path is precisely the coprojection $C \to (A \otimes B) \otimes C$, while the bottom path is precisely the coprojection $C \to A \otimes (B \otimes C)$.
\end{proof}

The following extends \cref{initiality-via-splitting} from nullary coproducts to finite coproducts.

\begin{lemma}
    \label{coproducts-via-splitting}
    Let $(\M, \otimes, I, \lambda, \rho)$ be a semicocartesian colax unital magmal category and suppose that the identity functor on $\M$ admits a unital magma structure $(\mu, \eta)$ in $[\M, \M]$. $(\M, \otimes, I, \lambda, \rho)$ is naturally weakly cocartesian\footnotemark{} and the following endomorphism is idempotent for each $A, B \in \M$.
	\begin{equation}
		\label{idempotent}
		A \otimes B \xto{\rho_A \otimes \lambda_B} (A \otimes I) \otimes (I \otimes B) \xto{(A \otimes \eta_B) \otimes (\eta_A \otimes B)} (A \otimes B) \otimes (A \otimes B) \xto{\mu_{A \otimes B}} A \otimes B
	\end{equation}
	Furthermore, this idempotent splits if and only if $\M$ admits the binary coproduct of $A$ and $B$, in which case the splitting exhibits the binary coproduct. In particular, $\M$ is cocartesian if and only if the diagram \eqref{quasi-symmetry-equation} commutes for each $A, B \in \M$.
	\footnotetext{That is, the tensor product $A \otimes B$ satisfies the existence and naturality property of the mediating morphisms from a binary coproduct, but not necessarily the uniqueness property.}%
\end{lemma}

\begin{proof}
	The result will follow from \cref{colimit-from-splitting-idempotent}, once we exhibit a section for the natural transformation
	\begin{equation}
		\label{coproduct-retraction}
		\M(A \otimes B, {-}) \to \M(A, {-}) \times \M(B, {-})
	\end{equation}
	given by precomposing the candidate coprojections of \cref{coprojections}. Suppose that $X \in \M$ is an object equipped with morphisms $a \colon A \to X$ and $b \colon B \to X$. The morphism
    \begin{equation}
        [a, b] \defeq A \otimes B \xto{a \otimes b} X \otimes X \xto{\mu_X} X
    \end{equation}
	defines a componentwise section for \eqref{coproduct-retraction}, the two coprojection laws for a weak coproduct following from commutativity of the following diagrams, where the unlabelled regions are unitality.
	\[
	\begin{tikzcd}
		A & X \\
		{A \otimes I} & {X \otimes I} & X \\
		{A \otimes B} & {X \otimes X}
		\arrow["a", from=1-1, to=1-2]
		\arrow[""{name=0, anchor=center, inner sep=0}, "{\rho_A}"', from=1-1, to=2-1]
		\arrow[""{name=1, anchor=center, inner sep=0}, "{\rho_X}"{description}, from=1-2, to=2-2]
		\arrow[equals, from=1-2, to=2-3]
		\arrow["{a \otimes I}"{description}, from=2-1, to=2-2]
		\arrow[""{name=2, anchor=center, inner sep=0}, "{A \otimes \eta_B}"', from=2-1, to=3-1]
		\arrow[""{name=3, anchor=center, inner sep=0}, "{X \otimes \eta_X}"{description}, from=2-2, to=3-2]
		\arrow["{a \otimes b}"', from=3-1, to=3-2]
		\arrow["{\mu_X}"', from=3-2, to=2-3]
		\arrow["{\nat\rho}"{description}, draw=none, from=0, to=1]
		\arrow["{\nat\eta}"{description}, draw=none, from=2, to=3]
	\end{tikzcd}
	\hspace{2em}
  \begin{tikzcd}
    & X & B \\
    X & {I \otimes X} & {I \otimes B} \\
    & {X \otimes X} & {A \otimes B}
    \arrow[equals, nfold, from=1-2, to=2-1]
    \arrow[""{name=0, anchor=center, inner sep=0}, "{\lambda_X}"{description}, from=1-2, to=2-2]
    \arrow["b"', from=1-3, to=1-2]
    \arrow[""{name=1, anchor=center, inner sep=0}, "{\lambda_B}", from=1-3, to=2-3]
    \arrow[""{name=2, anchor=center, inner sep=0}, "{\eta_X \otimes X}"{description}, from=2-2, to=3-2]
    \arrow["{I \otimes b}"{description}, from=2-3, to=2-2]
    \arrow[""{name=3, anchor=center, inner sep=0}, "{\eta_A \otimes B}", from=2-3, to=3-3]
    \arrow["{\mu_X}", from=3-2, to=2-1]
    \arrow["{a \otimes b}", from=3-3, to=3-2]
    \arrow["{\nat\lambda}"{description}, draw=none, from=1, to=0]
    \arrow["{\nat\eta}"{description}, draw=none, from=3, to=2]
  \end{tikzcd}
	\]
	Naturality of the section follows from naturality of $\mu$. It remains simply to observe that, by definition, \eqref{idempotent} is indeed the idempotent induced by the naturally weak coproduct.
\end{proof}

\begin{corollary}
	Let $(\M, \otimes, I, \lambda, \rho)$ be a colax unital magmal category and suppose that the identity functor on $\M$ admits a unital magma structure $(\mu, \eta)$ in $[\M, \M]$. If idempotents in $\M$ split, then $\M$ admits finite coproducts.
\end{corollary}

\begin{proof}
	Immediate from \cref{eta-I-is-1,coproducts-via-splitting}.
\end{proof}

\section{Symmetry}

Next, we consider the interaction between cocartesian structure and symmetry. First, we observe that, just as the universal property of coproducts in a cocartesian magmal category determines a unique associator, so too does it determine a unique symmetry.

\begin{lemma}
    For every cocartesian monoidal category, there is a unique natural transformation $\{ \sigma_{A, B} \colon A \otimes B \to B \otimes A \}_{A, B}$ satisfying the unit laws for a braiding. This natural transformation is moreover a symmetry.
\end{lemma}

\begin{proof}
  It is proven in \cite[\S3.2.42]{kock2004frobenius} that there is a unique symmetry, but the stronger statement holds. Let $\sigma$ be such a natural transformation. The unit law for the braiding implies that $\sigma_{A, 0} \colon A + 0 \to 0 + A$ is induced by the identity on $A$. Naturality of $\sigma$ implies that the following diagram commutes, and thus that $\sigma_{A, B} \c \copi_1 = \copi_2$.
  \[\begin{tikzcd}
    {A + B} & {B + A} \\
    {A + 0} & {0 + A}
    \arrow["{\sigma_{A, B}}", from=1-1, to=1-2]
    \arrow["{A + []_B}", from=2-1, to=1-1]
    \arrow["{\sigma_{A, 0}}"', from=2-1, to=2-2]
    \arrow["{[]_B + A}"', from=2-2, to=1-2]
  \end{tikzcd}\]
  The same argument implies that $\sigma_{A, B} \c \copi_2 = \copi_1$, and thus that $\sigma_{A, B}$ is uniquely determined to be the symmetry.
\end{proof}

We hence see that, under the assumption of symmetry, condition \eqref{quasi-symmetry-equation} may be replaced by another, involving the symmetry.

\begin{corollary}
    \label{symmetry-implies-quasi-symmetry}
    Let $(\M, \otimes, I, \lambda, \rho, \alpha, \sigma)$ be a braided monoidal category. Suppose that the identity functor on $\M$ admits a unital magma structure $(1_\M, \mu, \eta)$ in $[\M, \M]$. If the diagram \eqref{symmetry-equation} commutes for each $A, B \in \M$, then the diagram \eqref{quasi-symmetry-equation} commutes for each $A, B \in \M$.
\end{corollary}

\begin{proof}
    The following diagram commutes using the unit laws for the magmas, and the assumption. Commutativity of the region ($*$) follows from naturality of the braiding $\sigma$ and the unit law for the braiding.
	\[\begin{tikzcd}[column sep=large]
		{(A \otimes I) \otimes (I \otimes B)} && {(A \otimes B) \otimes (A \otimes B)} \\
		& {(A \otimes A) \otimes (B \otimes B)} \\
		{A \otimes B} && {A \otimes B}
		\arrow[""{name=0, anchor=center, inner sep=0}, "{(A \otimes \eta_B) \otimes (\eta_A \otimes B)}", from=1-1, to=1-3]
		\arrow["{(A \otimes \eta_A) \otimes (\eta_B \otimes B)}"{description}, from=1-1, to=2-2]
		\arrow["{\mu_{A \otimes B}}", from=1-3, to=3-3]
		\arrow["{A \otimes \sigma_{A, B} \otimes B}"{description}, from=2-2, to=1-3]
		\arrow["{\mu_A \otimes \mu_B}"{description}, from=2-2, to=3-3]
		\arrow["{\rho_A \otimes \lambda_B}", from=3-1, to=1-1]
		\arrow[equals, from=3-1, to=3-3]
		\arrow["{\text{($*$)}}"{description}, draw=none, from=0, to=2-2]
	\end{tikzcd}\qedshift\]
\end{proof}

Let us provide a word of explanation as to the nature of condition \eqref{symmetry-equation}: as the following lemma shows, it is precisely the condition that the chosen multiplication structure on a tensor product $A \otimes B$ is given by a certain, canonical choice.

\begin{lemma}[{\cite{fox1976coalgebras}}]
    \label{tensor-product-of-unital-magmas}
    In a symmetric monoidal category with unital magmas $(A, \mu_A, \eta_A)$ and $(B, \mu_B, \eta_B)$, the morphisms
	\[(A \otimes B) \otimes (A \otimes B) \xto{A \otimes \sigma_{B, A} \otimes B} (A \otimes A) \otimes (B \otimes B) \xto{\mu_A \otimes \mu_B} A \otimes B\]
	\[I \xto{\lambda_I} I \otimes I \xto{\eta_A \otimes \eta_B} A \otimes B\]
    define a unital magma structure on $A \otimes B$. Furthermore, if the two unital magmas are associative/commutative, so is their tensor product.
\end{lemma}

\begin{remark}[label=magmality-of-mu]
	To ensure a symmetric monoidal structure is cocartesian, it is common to see in the literature the imposition that $\mu \colon {\otimes} \c \Delta \tto 1_\M$ and $\eta \colon \u I \tto 1_\M$ be monoidal natural transformations (\eg~\cite[Corollary~17]{mellies2009categorical}). In particular, magmality of $\mu$ is precisely the commutativity of the diagram \eqref{symmetry-equation}, so these conditions are certainly sufficient to imply cocartesianness. However, they also contain a significant amount of redundancy (as has also been observed in the remark following \cite[Corollary~17]{mellies2009categorical}): namely, both unitality of $\mu$ and monoidality of $\eta$ are implied by the commutativity of the diagram \eqref{symmetry-equation}.
\end{remark}

\begin{remark}
	As mentioned in the previous remark, in the presence of a braiding, magmality of $\mu$ implies the commutativity of the diagram \eqref{symmetry-equation}. However, in the absence of a chosen braiding, there is, in general, no canonical magmal structure on the functor ${\otimes} \c \Delta$, and so it does not make sense to ask for $\mu$ to be magmal. However, we may compose the diagram \eqref{symmetry-equation} with the chosen unit structures to obtain the following diagram.
	\[\begin{tikzcd}[column sep=8em]
		{(A \otimes I) \otimes (I \otimes B)} & {(A \otimes A) \otimes (B \otimes B)} & {(A \otimes B) \otimes (A \otimes B)} \\
		{A \otimes B} && {A \otimes B}
		\arrow["{(A \otimes \eta_A) \otimes (\eta_B \otimes B)}"{description}, from=1-1, to=1-2]
		\arrow["{(A \otimes \eta_B) \otimes (\eta_A \otimes B)}", curve={height=-18pt}, from=1-1, to=1-3]
		\arrow["{A \otimes \sigma_{A, B} \otimes B}"{description}, from=1-2, to=1-3]
		\arrow["{\mu_A \otimes \mu_B}"{description}, from=1-2, to=2-3]
		\arrow["{\mu_{A \otimes B}}", from=1-3, to=2-3]
		\arrow["{\rho_A \otimes \lambda_B}", from=2-1, to=1-1]
		\arrow[equals, from=2-1, to=2-3]
	\end{tikzcd}\]
	Condition \eqref{quasi-symmetry-equation} is essentially an axiomatisation of the commutativity of the diagram above; it is the closest we can get to asking for $\mu$ to be a magmal natural transformation in the absence of symmetry.
\end{remark}

\section{An adjoint characterisation of cocartesianness}
\label{an-adjoint-characterisation}

In this section, we move to the second flavour of magmal characterisations of cocartesian monoidal categories. We will show that the existence of a right adjoint to the tensor product functor ${\otimes} \colon \M \times \M \to \M$ in a unital magmal category implies that $\M$ is cocartesian. Before proceeding with the proof, let us first spell out what it means for ${\otimes}$ to admit a right adjoint $\tp{L, R} \colon \M \to \M \times \M$. The most convenient characterisation for our purposes will be in terms of couniversal morphisms.

Spelling out the definition of such an adjoint, for every object $A \in \M$, we ask that there exists a morphism $\varepsilon_A \colon LA \otimes RA \to A$ such that every morphism $f \colon X \otimes Y \to A$ in $\M$ factors as $f^L \otimes f^R$ through $\varepsilon_A$, for a unique pair of morphisms $f^L \colon X \to LA$ and $f^R \colon Y \to RA$. It follows that these assignments are both natural~\cite[Theorem~IV.1.2]{maclane1998categories}.
\[\begin{tikzcd}[column sep=large]
	& {LA \otimes RA} \\
	{X \otimes Y} && A
	\arrow["{\varepsilon_A}", from=1-2, to=2-3]
	\arrow["{f^L \otimes f^R}", dashed, from=2-1, to=1-2]
	\arrow["f"', from=2-1, to=2-3]
\end{tikzcd}\]

We may thus define natural transformations
\begin{align*}
	(\lambda\inv)^L & \colon \u I \tto L &
	(\lambda\inv)^R & \colon \ph \tto R &
	(\rho\inv)^L & \colon \ph \tto L &
	(\rho\inv)^R & \colon \u I \tto R
\end{align*}
by the couniversal property, as follows.
\begin{equation}
\label{couniversal-factorisation}
\begin{tikzcd}
	& {L \otimes R} \\
	{\u I \otimes \ph} && \ph
	\arrow["\varepsilon", from=1-2, to=2-3]
	\arrow["{(\lambda\inv)^L \otimes (\lambda\inv)^R}", from=2-1, to=1-2]
	\arrow["{\lambda\inv}"', from=2-1, to=2-3]
\end{tikzcd}
\hspace{4em}
\begin{tikzcd}
	& {L \otimes R} \\
	{\ph \otimes \u I} && \ph
	\arrow["\varepsilon", from=1-2, to=2-3]
	\arrow["{(\rho\inv)^L \otimes (\rho\inv)^R}", from=2-1, to=1-2]
	\arrow["{\rho\inv}"', from=2-1, to=2-3]
\end{tikzcd}
\end{equation}
This allows us to construct natural transformations as follows.
\begin{align}
	\eta & \defeq \u I \xto{\u{\lambda_I}} \u{I \otimes I} \xto{(\lambda\inv)^L \otimes (\rho\inv)^R} L \otimes R \xto{\varepsilon} \ph \\
	\mu & \defeq \ph \otimes \ph \xto{(\rho\inv)^L \otimes (\lambda\inv)^R} L \otimes R \xto{\varepsilon} \ph
\end{align}

To proceed, we could make use of our magmal characterisation of cocartesian monoidal categories, by showing that $(\mu, \eta)$ exhibits a unital magma structure on $1_\M$ satisfying condition \eqref{quasi-symmetry-equation}. However, there is a more direct approach.

\begin{lemma}
    \label{adjoint-implies-cocartesian}
    Let $(\M, \otimes, I, \lambda, \rho)$ be a unital magmal category for which ${\otimes} \colon \M \times \M \to \M$ admits a right adjoint $\tp{L, R} \colon \M \to \M \times \M$. Then the unital magmal structure is cocartesian monoidal.
\end{lemma}

\begin{proof}
	We first show that $I$ is initial using \cref{initiality-via-splitting}, for which it suffices to show that $\eta_I$ is the identity. This follows from commutativity of the following diagram, which uses the commutativity of the leftmost triangle in \eqref{couniversal-factorisation}, and that $\lambda_I = \rho_I$.
	\[\begin{tikzcd}[column sep=huge]
		{I \otimes I} && {LI \otimes RI} \\
		I && I
		\arrow["{(\lambda\inv_I)^L \otimes (\rho\inv_I)^R}", curve={height=-12pt}, from=1-1, to=1-3]
		\arrow["{(\lambda\inv_I)^L \otimes (\lambda\inv_I)^R}"{description}, from=1-1, to=1-3]
		\arrow["{\lambda\inv_I}"{description}, from=1-1, to=2-3]
		\arrow["{\varepsilon_I}", from=1-3, to=2-3]
		\arrow["{\lambda_I}", from=2-1, to=1-1]
		\arrow[equals, from=2-1, to=2-3]
	\end{tikzcd}\]
  Since $I$ is initial, $(A, B)$ is the coproduct $(A, I) + (I, B)$ in $\M \times \M$, which is preserved by the left adjoint ${\otimes} \colon \M \times \M \to \M$. Precomposing the coprojections by $\rho_A$ and $\lambda_B$ respectively, we obtain precisely the cospan of \cref{coprojections}, witnessing the unital magmal structure as cocartesian.
\end{proof}

We make note of a useful consequence of this characterisation.

\begin{corollary}
    A cartesian monoidal category $\M$ has finite biproducts if and only if the cartesian product functor $\ph \times \ph \colon \M \times \M \to \M$ admits a right adjoint.
\end{corollary}

\begin{proof}
    One direction is trivial. In the other, if ${\times}$ admits a right adjoint, it is necessarily the diagonal functor, in which case ${\times}$ exhibits the coproduct.
\end{proof}

\section{Proof of theorem and necessity of the conditions}
\label{necessity}

In this final section, we tie everything together, giving a proof of the main theorem, and counterexamples demonstrating that condition \eqref{quasi-symmetry-equation}, unitality, and the assumption that every morphism is a homomorphism are necessary.

\begin{proof}[Proof of \cref{characterisations}]
    We will establish the following implications.
	\[\begin{tikzcd}
		{\text{\ref{adjoint-characterisation}}} & {\text{\ref{cocartesian}}} & {\text{\ref{uniqueness-characterisation}}} \\
		& {\text{\ref{sharpest-characterisation}}} & {\text{\ref{symmetry-characterisation}}}
		\arrow[curve={height=12pt}, from=1-1, to=1-2]
		\arrow[curve={height=12pt}, from=1-2, to=1-1]
		\arrow[from=1-2, to=1-3]
		\arrow[from=1-3, to=2-3]
		\arrow[from=2-2, to=1-2]
		\arrow[from=2-3, to=2-2]
	\end{tikzcd}\]
    (\ref{cocartesian} $\implies$ \ref{uniqueness-characterisation}) is \cref{cartesian-implies-unique-commutative-comonoid}.
    (\ref{uniqueness-characterisation} $\implies$ \ref{symmetry-characterisation}) follows from \cref{tensor-product-of-unital-magmas}.
    (\ref{symmetry-characterisation} $\implies$ \ref{sharpest-characterisation}) is \cref{symmetry-implies-quasi-symmetry}.
    (\ref{sharpest-characterisation} $\implies$ \ref{cocartesian}) follows from \cref{eta-I-is-1,coproducts-via-splitting}.
    (\ref{cocartesian} $\implies$ \ref{adjoint-characterisation}) follows since ${+} \adj \Delta$.
    (\ref{adjoint-characterisation} $\implies$ \ref{cocartesian}) is \cref{adjoint-implies-cocartesian}.
\end{proof}

We show that condition \eqref{quasi-symmetry-equation} is indeed necessary by constructing an example of a non-cocartesian symmetric monoidal category $\M$ for which $1_\M$ admits the structure of a commutative monoid, but for which commutativity of the diagram \eqref{quasi-symmetry-equation} does not hold.

\begin{proposition}
    \label{semigroup-monoidal-structure}
    Let $(\V, \diamond, I)$ be a symmetric monoidal category with finite coproducts that are preserved by tensoring on either side. $\V$ is equipped with a semicocartesian symmetric monoidal structure whose tensor product is given by:
    \[A \otimes B \defeq A + (A \diamond B) + B\]
    Furthermore, the category of [commutative] monoids in $(\V, \otimes, 0)$ is the category of [commutative] semigroups in $(\V, \diamond, I)$.
\end{proposition}

\begin{proof}
  The existence of the symmetric monoidal structure follows from the construction of \cite[Propositions~23 \& 27]{garner2023hypernormalisation} with respect to the trivial symmetric tricocycloid structure on $I$. Consider a monoid $M$ with respect to this monoidal structure. The left unit law (below), together with the right unit law, means that the multiplication $m \colon M + M \diamond M + M \to M$ is determined by its central component: the other two components are induced by the identity on $M$.
  \[\begin{tikzcd}[column sep=7em]
    {0 + 0 \diamond M + M} & {M + M \diamond M + M} \\
    & M
    \arrow["{[]_M + []_{M \diamond M} + M}", from=1-1, to=1-2]
    \arrow["\iso"', from=1-1, to=2-2]
    \arrow["m", from=1-2, to=2-2]
  \end{tikzcd}\]
	Observe that: \[(M \otimes M) \otimes M = (M \otimes M) + (M \otimes M) \diamond M + M = (M + M \diamond M + M) + (M + M \diamond M + M) \diamond M + M\]
	The associativity law is given by:
  \[\begin{tikzcd}
    {(M \otimes M) + (M \otimes M) \diamond M + M} && {M + M \diamond (M \otimes M) + (M \otimes M)} \\
    {M + M \diamond M + M} && {M \otimes M} \\
    & M
    \arrow["\alpha", from=1-1, to=1-3]
    \arrow["{m + m \diamond M + M}"', from=1-1, to=2-1]
    \arrow["{M + M \diamond m + m}", from=1-3, to=2-3]
    \arrow["m"', from=2-1, to=3-2]
    \arrow["m", from=2-3, to=3-2]
  \end{tikzcd}\]
	It is clear by inspection that the only nontrivial component is the $M \diamond M \diamond M \to M$ component, which is exactly the associativity law for a semigroup. Likewise, the only nontrivial component of the commutativity law is the $M \diamond M \to M$ component, which is exactly the commutativity law for a commutative semigroup.
  \[\begin{tikzcd}
    {M + M \diamond M + M} && {M + M \diamond M + M} \\
    & M
    \arrow["{[\copi_3, \copi_2 \sigma_{M, M}, \copi_1]}", from=1-1, to=1-3]
    \arrow["m"', from=1-1, to=2-2]
    \arrow["m", from=1-3, to=2-2]
  \end{tikzcd}\]
  The characterisation of the morphisms follows similarly.
\end{proof}

\begin{proposition}
  Condition \eqref{quasi-symmetry-equation} is necessary.
\end{proposition}

\begin{proof}
  We give two proofs. The first is due to Egger and has the virtue of being mildly simpler. The second establishes the stronger statement that condition \eqref{quasi-symmetry-equation} is necessary even in the presence of commutativity.
  \begin{enumerate}
    \item Consider the symmetric monoidal structure induced by instantiating \cref{semigroup-monoidal-structure} with respect to $(\Set, \times, 1)$. Every set is a monoid therein, since each set $A$ is equipped with semigroup structure given by the left band $(A, \pi_1)$ (alternatively the right band $(A, \pi_2)$). This assignment satisfies all of the conditions of \hyperref[sharpest-characterisation]{\cref*{characterisations}.\ref*{sharpest-characterisation}} but \eqref{quasi-symmetry-equation}: the longer path appearing in the condition sends a pair $(a, b)$ in the middle summand to its first component $a$ in the first summand, and thus \eqref{quasi-symmetry-equation} fails.
    \item Consider the symmetric monoidal structure $\otimes$ induced by instantiating \cref{semigroup-monoidal-structure} with respect to the tensor product $\diamond$ of bounded meet semilattices, which represents morphisms preserving finite meets in each variable separately. Every bounded meet semilattice is a monoid therein, since each bounded meet semilattice $A$ is equipped with commutative semigroup structure given by the zero morphism $0 \colon A \diamond A \to A$. This assignment satisfies all of the conditions of \hyperref[sharpest-characterisation]{\cref*{characterisations}.\ref*{sharpest-characterisation}} but \eqref{quasi-symmetry-equation}: taking $A = B = \{ 0 < 1 \}$, the longer path appearing in the condition preserves the outer summands, but is the zero morphism on the inner summand, and thus \eqref{quasi-symmetry-equation} fails.
    \qedhere
  \end{enumerate}
\end{proof}

\begin{remark}
  Both unit laws are necessary for the magmas in \hyperref[sharpest-characterisation]{\cref*{characterisations}.\ref*{sharpest-characterisation}} The category of pointed sets is equipped with a colax unital magmal structure, in which the tensor product is given by left-projection $f \otimes g \defeq f$, the unit $I \defeq 1$ by the one-element set, the left unitor $\lambda_A \colon A \to I \otimes A = I$ by the unique morphism, and the right unitor $\rho_A \colon A \to A \otimes I = A$ by the identity morphism. Every object $A$ is equipped with a unique morphism $\eta_A \colon I \to A$. Taking $\mu_A \colon A \otimes A \to A$ to be the identity morphism, we have that every object is equipped with a right-unital magma structure for which every morphism is a homomorphism, and for which condition \eqref{quasi-symmetry-equation} holds. However, this colax unital magmal structure is evidently not cocartesian. The symmetric argument shows the right unit law to be necessary.

	(See \cite[\S3]{roberts2023substructural} for further examples of (non-cocartesian) magmal categories in which objects are equipped with (non-unital) magma structure for which morphisms are homomorphisms.)
\end{remark}

\begin{remark}
  Naturality of $\mu$ and of $\eta$ is necessary in \hyperref[sharpest-characterisation]{\cref*{characterisations}.\ref*{sharpest-characterisation}}
  \begin{itemize}
    \item For naturality of $\mu$, consider the category of bounded meet semilattices and functions preserving the top element but not necessarily the binary meets. The cartesian product of bounded meet semilattices exhibits a monoidal category satisfying every assumption except for naturality of~$\mu$.
    \item For naturality of $\eta$, consider the category of bounded meet semilattices and functions preserving binary meets but not necessarily the top element (\aka{} stable homomorphisms~\cite[\S2.1.3.12]{cockett2002restriction}). The cartesian product of bounded meet semilattices exhibits a monoidal category satisfying every assumption except for naturality of $\eta$.
  \end{itemize}
  (See \cite[\S6.6]{corradini1999algebraic} for a survey of various classes of (non-cocartesian) monoidal categories, in which objects are equipped with algebraic structure, but for which morphisms are not necessarily homomorphisms.)
\end{remark}

\begin{remark}
  \label{further-generalisation}
  There is, in fact, one further mild generalisation of the main theorem: we need only require naturality of \emph{one} of the unitors $\lambda$ or $\rho$ in \hyperref[sharpest-characterisation]{\cref*{characterisations}.\ref*{sharpest-characterisation}} To see this, observe that the only places in which naturality of the unitors is invoked for the main theorem are in the proof of \cref{eta-I-is-1}, which only uses naturality of $\lambda$; and in the proof of \cref{coproducts-via-splitting}, in which the commutative diagram involving naturality of $\rho$ may be replaced by the following diagram, where the unlabelled region is unitality.
  \[\begin{tikzcd}
    A && A & X \\
    {A \otimes I} && {A \otimes A} \\
    & {X \otimes I} \\
    {A \otimes B} &&& {X \otimes X}
    \arrow[equals, nfold, from=1-1, to=1-3]
    \arrow["{\rho_A}"', from=1-1, to=2-1]
    \arrow["a"{description}, from=1-3, to=1-4]
    \arrow["{A \otimes \eta_A}"{description}, from=2-1, to=2-3]
    \arrow["{a \otimes I}"{description}, from=2-1, to=3-2]
    \arrow["{A \otimes \eta_B}"', from=2-1, to=4-1]
    \arrow["{\mu_A}"{description}, from=2-3, to=1-3]
    \arrow["{\nat\mu}"{description}, draw=none, from=2-3, to=1-4]
    \arrow["{a \otimes a}"{description}, from=2-3, to=4-4]
    \arrow["{\nat\eta}"{description}, draw=none, from=3-2, to=2-3]
    \arrow["{X \otimes \eta_X}"{description}, from=3-2, to=4-4]
    \arrow["{\nat\eta}"{description}, draw=none, from=4-1, to=3-2]
    \arrow["{a \otimes b}"', from=4-1, to=4-4]
    \arrow["{\mu_X}"', from=4-4, to=1-4]
  \end{tikzcd}\]
  Naturality of $\rho$ follows. Explicitly, $\rho_A \colon A \to A \otimes I$ is the coprojection of $A$ into $A + 0$, which has inverse $\mu_A (A \otimes \eta_A)$: since this is natural in $A$, so is $\rho_A$.

  It is not possible to drop naturality of both unitors simultaneously. The category of bounded meet semilattices with the cartesian binary product forms a magmal category. For the candidate unit, take the two-element poset $I \defeq \{ 0 < 1 \}$. For each semilattice $A \neq I$, take
  \begin{equation}
    \lambda_A \defeq (1, {-}) \colon A \to I \times A \hspace{4em} \rho_A \defeq ({-}, 1) \colon A \to A \times I
  \end{equation}
  and take $\lambda_I = \rho_I$ to be the diagonal. Neither family of morphisms is natural with respect to the unique morphism $\unit \colon I \to 1$.

  Every semilattice $A$ is equipped with a unital magma structure in which $\eta_A \colon I \to A$ is constant at $\top$ and in which $\mu_A \colon A \times A \to A$ is the meet. This assignment satisfies all of the assumptions of \hyperref[sharpest-characterisation]{\cref*{characterisations}.\ref*{sharpest-characterisation}} save for the naturality of the unitors. However, $I$ is not initial.
\end{remark}

\printbibliography

\end{document}

%% file: core.tex
\newcommand{\ifarticle}[2]{
    \csname@ifclassloaded\endcsname{beamer}{#2}{#1}
}

\newcommand{\ifbook}[2]{
    \csname@ifclassloaded\endcsname{amsbook}{#1}{#2}
}

    \usepackage[T1]{fontenc} 
    \usepackage{xparse} 
    \usepackage{xspace} 
    \usepackage{etoolbox} 
    \usepackage{xpatch} 

    \usepackage{amsmath,amssymb} 
    \usepackage{keytheorems} 
    \allowdisplaybreaks[1]

    \usepackage{mathtools} 
    \usepackage{mathrsfs} 
    \usepackage{stmaryrd} 
    \usepackage{bbm} 
    \usepackage{quiver} 
    \usepackage{xcolor} 
    \usepackage{scalerel} 
    \usepackage{booktabs} 
    \usepackage{nameref} 
    \usepackage[british]{babel} 
    \usepackage{csquotes} 
    \usepackage[UKenglish]{isodate} 
    \usepackage{microtype} 
    \usepackage[splitrule]{footmisc} 

    \ifarticle{
        \PassOptionsToPackage{hyphens}{url}
        \usepackage[pdfusetitle,linktoc=all,bookmarksdepth=2]{hyperref} 

        \usepackage{zref-clever} 
        \zcsetup{nameinlink,noabbrev,cap=true}
        \newcommand{\cref}{\zcref}
        \newcommand{\Cref}{\zcref[cap=true]}
        \usepackage{footnotebackref} 

        \usepackage[shortlabels]{enumitem} 
        \setlist{topsep=2pt,itemsep=2pt,partopsep=2pt,parsep=2pt} 

        \setcounter{tocdepth}{1}
    }{}
    \usepackage[backend=biber,style=alphabetic,sorting=nyt,abbreviate=false,backref=true,backrefstyle=three,maxnames=9,minalphanames=3,maxalphanames=4]{biblatex} 
    \DeclareUnicodeCharacter{0301}{\TODO[Invalid symbol.]} 

    \DeclareFieldFormat*{citetitle}{\emph{#1}} 
    \DeclareCiteCommand{\citetitle}
        {\boolfalse{citetracker}%
        \boolfalse{pagetracker}%
        \usebibmacro{prenote}}
        {\ifciteindex
            {\indexfield{indextitle}}
            {}%
        \printtext[bibhyperref]{\printfield[citetitle]{labeltitle}}}
        {\multicitedelim}
        {\usebibmacro{postnote}}
    \DeclareCiteCommand*{\citeauthor}
        {\defcounter{maxnames}{99}%
        \defcounter{minnames}{99}%
        \defcounter{uniquename}{2}%
        \boolfalse{citetracker}%
        \boolfalse{pagetracker}%
        \usebibmacro{prenote}}
        {\ifciteindex{\indexnames{labelname}}{}%
        \printnames{labelname}}
        {\multicitedelim}
        {\usebibmacro{postnote}}

    \usepackage{calc} 
    \usepackage{tabularray} 
    \usepackage{ebproof} 
    \usepackage{forest} 
    \usepackage{adjustbox} 
    \usepackage{float} 
    \usepackage{stfloats}

    \makeatletter
    \ifarticle{

        \@ifclassloaded{amsart}{
        \def\subsection{\@startsection{subsection}{2}%
          \z@{.5\linespacing\@plus.7\linespacing}{.5\baselineskip}%
          {\normalfont\centering\bfseries}}

        \def\paragraph{\@startsection{paragraph}{4}%
          \z@\z@{-\fontdimen2\font}%
          {\normalfont\bfseries}}
        }{}

        \def\@makefntext{\indent\@makefnmark}
    }{}
    \makeatother

    \ifarticle{
        \theoremstyle{plain}
        \ifbook{
            \newkeytheorem{theorem}[name=Theorem, parent=chapter, refname={theorem, theorems}, Refname={Theorem, Theorems}]
        }{
            \newkeytheorem{theorem}[name=Theorem, parent=section, refname={theorem, theorems}, Refname={Theorem, Theorems}]
        }

        \newkeytheorem{theorem*}[name=Theorem, numbered=false]

        \newkeytheorem{corollary}[name=Corollary, sibling=theorem, refname={corollary, corollaries}, Refname={Corollary, Corollaries}]
        \newkeytheorem{lemma}[name=Lemma, sibling=theorem, refname={lemma, lemmas}, Refname={Lemma, Lemmas}]
        \newkeytheorem{proposition}[name=Proposition, sibling=theorem, refname={proposition, propositions}, Refname={Proposition, Propositions}]

        \theoremstyle{definition}

        \newkeytheorem{axiom}[name=Axiom, sibling=theorem, qed=$\lrcorner$, refname={axiom, axioms}, Refname={Axiom, Axioms}]
        \newkeytheorem{conjecture}[name=Conjecture, sibling=theorem, qed=$\lrcorner$, refname={conjecture, conjectures}, Refname={Conjecture, Conjectures}]
        \newkeytheorem{construction}[name=Construction, sibling=theorem, qed=$\lrcorner$, refname={construction, constructions}, Refname={Construction, Constructions}]
        \newkeytheorem{definition}[name=Definition, sibling=theorem, qed=$\lrcorner$, refname={definition, definitions}, Refname={Definition, Definitions}]
        \newkeytheorem{example}[name=Example, sibling=theorem, qed=$\lrcorner$, refname={example, examples}, Refname={Example, Examples}]
        \newkeytheorem{notation}[name=Notation, sibling=theorem, qed=$\lrcorner$]
        \newkeytheorem{note}[name=Note, sibling=theorem, qed=$\lrcorner$, refname={note, notes}, Refname={Note, Notes}]
        \newkeytheorem{remark}[name=Remark, sibling=theorem, qed=$\lrcorner$, refname={remark, remarks}, Refname={Remark, Remarks}]

        \newenvironment{sketch}{\proof}{\endproof}
    }{}

    \ExplSyntaxOn
    \clist_map_inline:nn { axiom, conjecture, construction, corollary, definition, example, lemma, notation, note, proposition, remark, theorem } {
        \AtBeginEnvironment{#1} {
            \zcsetup{countertype={enumi=#1, enumii=#1}}
            \setlist[enumerate,1]{ref={\csname the#1\endcsname.(\arabic*)}}
            \setlist[enumerate,2]{ref={\csname the#1\endcsname.(\arabic*).(\alph*)}}
        }
    }
    \ExplSyntaxOff

    \newcommand{\qedshift}{\vspace*{-\baselineskip}}

    \ExplSyntaxOn
    \NewDocumentCommand{\mathcommand}{mO{0}m}
     {
      \exp_args:Nc \NewCommandCopy {khue_kept_\cs_to_str:N #1} { #1 }
      \exp_args:Nc \newcommand {khue_new_\cs_to_str:N #1}[#2]{#3}
      \DeclareDocumentCommand {#1} {}
       {
        \mode_if_math:TF
         {
          \use:c {khue_new_\cs_to_str:N #1}
         }
         {
          \use:c {khue_kept_\cs_to_str:N #1}
         }
       }
     }
    \ExplSyntaxOff

    \newsavebox\tikzcdbox

    \mathcommand{\h}{\textup{-}}

    \newcommand{\tx}{\mathrm}
    \mathcommand{\b}{\mathbf}
    
    \newcommand{\cl}{\mathcal}
    \mathcommand{\bb}{\mathbb}
    \DeclareMathAlphabet{\bbn}{U}{bbold}{m}{n}
    
    \mathcommand{\sf}{\mathsf}
    \mathcommand{\u}{\underline}
    \mathcommand{\o}{\overline}

    \newcommand{\TODO}[1][TODO]{\textcolor{orange}{\textup{#1}}\xspace}

    \newcommand{\flip}[1]{\text{\rotatebox[origin=c]{-180}{$#1$}}}
    
    \newcommand{\defeq}{\mathrel{:=}}
    \mathcommand{\d}{\mathbin{;}}
    \mathcommand{\c}{\circ}
    \newcommand{\ph}[1][]{{({-}_{#1})}}
    
    \newcommand{\iso}{\cong}
    
    \newcommand{\xto}{\xrightarrow}
    \newcommand{\xfrom}{\xleftarrow}
    \newcommand{\tto}{\Rightarrow}
    
    \newcommand{\xtto}{\xRightarrow}
    
    \newcommand{\epito}{\twoheadrightarrow}
    
    \newcommand*\cocolon{%
            \nobreak
            \mskip6mu plus1mu
            \mathpunct{}%
            \nonscript
            \mkern-\thinmuskip
            {:}%
            \mskip2mu
            \relax
    }

    \makeatletter
    \def\slashedarrowfill@#1#2#3#4#5{%
    $\m@th\thickmuskip0mu\medmuskip\thickmuskip\thinmuskip\thickmuskip
    \relax#5#1\mkern-7mu%
    \cleaders\hbox{$#5\mkern-2mu#2\mkern-2mu$}\hfill
    \mathclap{#3}\mathclap{#2}%
    \cleaders\hbox{$#5\mkern-2mu#2\mkern-2mu$}\hfill
    \mkern-7mu#4$%
    }
    \def\rightslashedarrowfill@{%
    \slashedarrowfill@\relbar\relbar\mapstochar\rightarrow}
    \newcommand\xslashedrightarrow[2][]{%
    \ext@arrow 0055{\rightslashedarrowfill@}{#1}{#2}}
    \def\leftslashedarrowfill@{%
    \slashedarrowfill@\leftarrow\relbar\mapsfromchar\relbar}
    \newcommand\xslashedleftarrow[2][]{%
    \ext@arrow 0055{\leftslashedarrowfill@}{#1}{#2}}
    \def\rightdoubleslashedarrowfill@{%
    \slashedarrowfill@\relbar\relbar{\mapstochar\mkern-2mu\mapstochar}\rightarrow}
    \newcommand\xdoubleslashedrightarrow[2][]{%
    \ext@arrow 0055{\rightdoubleslashedarrowfill@}{#1}{#2}}
    \makeatother

    \newcommand{\inv}{^{-1}}

    \newcommand{\tp}[1]{\langle#1\rangle}
    \newcommand{\unit}{{\tp{}}}
    
    \newcommand{\adj}{\dashv}
    
    \DeclareFontFamily{U}{min}{}
    \DeclareFontShape{U}{min}{m}{n}{<-> udmj30}{}

    \mathcommand{\comma}{\downarrow}

    \newcommand{\copi}{\flip\pi}
    
    \newcommand{\ostar}{\mathbin{\scalerel*{\kern.075em\tikz{\node[draw, circle, line width=.36pt, inner sep=0]{\ensuremath{*}};}\kern.075em}{\odot}}}
    \newcommand{\bulletstar}{\mathbin{\scalerel*{\kern.075em\tikz{\node[fill, text=white, circle, line width=.36pt, inner sep=0]{\ensuremath{*}};}\kern.075em}{\odot}}}

    \makeatletter
    \def\widebreve{\mathpalette\wide@breve}
    \def\wide@breve#1#2{\sbox\z@{$#1#2$}%
         \mathop{\vbox{\m@th\ialign{##\crcr
    \kern0.08em\brevefill#1{0.8\wd\z@}\crcr\noalign{\nointerlineskip}%
                        $\hss#1#2\hss$\crcr}}}\limits}
    \def\brevefill#1#2{$\m@th\sbox\tw@{$#1($}%
      \hss\resizebox{#2}{\wd\tw@}{\rotatebox[origin=c]{90}{\upshape(}}\hss$}
    \makeatother

    \makeatletter
    \def\widearch{\mathpalette\wide@arch}
    \def\wide@arch#1#2{\sbox\z@{$#1#2$}%
         \mathop{\vbox{\m@th\ialign{##\crcr
    \kern0.08em\archfill#1{0.8\wd\z@}\crcr\noalign{\nointerlineskip}%
                        $\hss#1#2\hss$\crcr}}}\limits}
    \def\archfill#1#2{$\m@th\sbox\tw@{$#1($}%
      \hss\resizebox{#2}{\wd\tw@}{\rotatebox[origin=c]{-90}{\upshape(}}\hss$}
    \makeatother

    \NewDocumentCommand{\jrule}{om}{%
        \IfNoValueTF{#1}
            {\textsc{#2}}
            {$#1$-\textsc{#2}}%
    }

    \newcommand{\Set}{{\b{Set}}}
    
    \newcommand{\xth}{\textsuperscript{th}}

    \newcommand{\eg}{e.g.\@\xspace}
    \newcommand{\ie}{i.e.\@\xspace}
    
    \newcommand{\cf}{cf.\@\xspace}
    
    \newcommand{\aka}{a.k.a.\@\xspace}
    \NewDocumentCommand{\etc}{t.}{etc.\@\xspace}
    \NewDocumentCommand{\ibid}{t.}{ibid.\@\xspace}
    \NewDocumentCommand{\loccit}{t.}{loc.\ cit.\@\xspace}

\makeatletter

\define@key{beamerframe}{c}[true]{
    \beamer@frametopskip=0pt plus 1fill\relax%
    \beamer@framebottomskip=0pt plus 1fill\relax%
}

\patchcmd{\beamer@sectionintoc}{\vfill}{\vskip\itemsep}{}{}

\makeatother

\ifarticle{}{

  \colorlet{colour-bg}{black!85} 
  \definecolor{colour-primary}{HTML}{cc80ff} 
  \colorlet{colour-text}{black!10} 
  \colorlet{colour-subtle}{black!40} 
  \colorlet{colour-block-bg}{black!80} 
  \definecolor{colour-warning-bg}{HTML}{ffea80} 
  \definecolor{colour-warning-primary}{HTML}{e08152} 

  \hypersetup{linktoc=all,colorlinks, citecolor={colour-primary}, linkcolor={colour-primary}, urlcolor={colour-primary}} 
  \urlstyle{sf}

  \usepackage{changepage} 

  \usepackage{ragged2e} 

  \setbeamertemplate{section in toc}[sections numbered]

  \apptocmd{\frame}{}{\justifying}{}

  \usefonttheme[onlymath]{serif}

  \beamertemplatenavigationsymbolsempty
  \setbeamertemplate{footline}{
    \hfill%
    \usebeamercolor[fg]{page number in head/foot}%
    \usebeamerfont{page number in head/foot}%
    \setbeamertemplate{page number in head/foot}[pagenumber]%
    \usebeamertemplate*{page number in head/foot}\kern.2cm\vskip.2cm%
  }

  \setbeamertemplate{frametitle}[default][center]
  \addtobeamertemplate{frametitle}{\vspace*{1ex}}{\vspace*{1ex}}

  \setbeamertemplate{itemize item}{$\bullet$}

  \setbeamertemplate{theorems}[numbered]
  \newtheorem{proposition}[theorem]{\translate{Proposition}}

  \addtobeamertemplate{block begin}
  {}
  {\vspace{0ex} 
  \begin{adjustwidth}{1em}{1em} 
  \tikzcdset{background color=colour-block-bg}
  }
  \addtobeamertemplate{block end}{\end{adjustwidth}%
  \vspace{1ex}} 
  {}
  \renewenvironment<>{block}[1]{%
      \begin{actionenv}#2%
        \def\insertblocktitle{\begin{adjustwidth}{.8em}{.8em}#1\end{adjustwidth}}%
        \par%
        \usebeamertemplate{block begin}}
      {\par%
        \usebeamertemplate{block end}%
      \end{actionenv}}

  \addtobeamertemplate{block example begin}
  {}
  {\vspace{0ex} 
  \begin{adjustwidth}{1em}{1em} 
  \tikzcdset{background color=colour-block-bg}
  }
  \addtobeamertemplate{block example end}{\end{adjustwidth}%
  \vspace{1ex}} 
  {}
  \renewenvironment<>{exampleblock}[1]{%
      \begin{actionenv}#2%
          \def\insertblocktitle{\begin{adjustwidth}{.8em}{.8em}#1\end{adjustwidth}}%
          \par%
          \only<presentation>{
            \setbeamercolor{local structure}{parent=example text}}%
          \usebeamertemplate{block example begin}}
        {\par%
          \usebeamertemplate{block example end}%
        \end{actionenv}}

    \setbeamercolor{background canvas}{bg=colour-bg}
    \tikzcdset{background color=colour-bg}

    \setbeamercolor{block title}{fg=colour-primary,bg=colour-block-bg}
    \setbeamercolor{block title example}{fg=colour-primary,bg=colour-block-bg}
    \setbeamercolor{block body}{bg=colour-block-bg}
    \setbeamercolor{block body example}{bg=colour-block-bg}

    \setbeamercolor{title}{fg=colour-primary}
    \setbeamercolor{frametitle}{fg=colour-primary}
    \setbeamercolor{alerted text}{fg=colour-primary}

    \setbeamercolor{normal text}{fg=colour-text}
    \setbeamercolor{item}{fg=colour-text}
    \setbeamercolor{section in toc}{fg=colour-text}

    \setbeamercolor{page number in head/foot}{fg=colour-subtle}

    \setbeamercolor*{bibliography entry author}{fg=colour-text}
    \setbeamercolor*{bibliography entry note}{fg=colour-text}

}